\documentclass{article}

\usepackage{fullpage}
\usepackage{palatino}

\newcommand\ifwithlilypond[2]{#2}
\newcommand\ifwithps[2]{#2}

\usepackage{amssymb,amsmath,amsthm}
\usepackage{graphics}
\ifwithps{\usepackage{pst-plot}
\usepackage{tikz}\usetikzlibrary{turtle}}{}
\usepackage[utf8]{inputenc}
\newtheorem{theorem}{Theorem}
\newtheorem{lemma}[theorem]{Lemma}
\theoremstyle{definition}
\newtheorem{example}[theorem]{Example}
\newtheorem{remark}[theorem]{Remark}
\usepackage[scaled]{helvet}
\usepackage{url}
\usepackage{ifthen}

\renewcommand\leq\leqslant
\renewcommand\geq\geqslant

\author{Maria Bras-Amorós}

\ifwithps{
\newcommand\teclablanca[2]{\filldraw[thick,black,fill=black!#2](#1,5.)[turtle={right=90,forward=2,right=90,forward=5.,right=90,forward=2,right=90,forward=5.}];}
\newcommand\teclanegra[2]{\filldraw[thick,black,fill=black!#2](#1,5.)[turtle={right=90,forward=1.6,right=90,forward=3.,right=90,forward=1.6,right=90,forward=3.}];}

\newcommand\keyboardI[8]{%
    \def\tempa{#1}%
    \def\tempb{#2}%
    \def\tempc{#3}%
    \def\tempd{#4}%
    \def\tempe{#5}%
    \def\tempf{#6}%
    \def\tempg{#7}%
    \def\temph{#8}%
}
\newcommand\keyboardII[5]{
  \noindent\resizebox{.9\textwidth}{!}{\begin{tikzpicture}
\teclablanca{0}{0}
\teclablanca{2}{0}
\teclablanca{4}{0}
\teclablanca{6}{0}
\teclablanca{8}{0}
\teclablanca{10}{0}
\teclablanca{12}{0}
\teclablanca{14}{0}

\teclanegra{1.2}{20}
\teclanegra{3.2}{20}
\teclanegra{7.2}{20}
\teclanegra{9.2}{20}
\teclanegra{11.2}{20}

\node at (1,1.) {\large  \tempa};
\node at (2.,3.) {{\large \tempb}};
\node at (3,1.) {\large \tempc};
\node at (4.,3.) {{\large \tempd}};
\node at (5.,1.) {\large \tempe};
\node at (7.,1.) {\large \tempf};
\node at (8.,3.) {{\large \tempg}};
\node at (9.,1.) {\large \temph};
\node at (10.,3.) {{\large #1}};
\node at (11.,1.) {\large #2};
\node at (12.,3.) {{\large #3}};
\node at (13.,1.) {\large #4};
\node at (15.,1.) {\large #5};
\end{tikzpicture}}}
}

\date{\today}

\title{Increasingly Enumerable Submonoids of ${\mathbb R}$: \\Music Theory as a Unifying Theme}
\usepackage{subcaption}

\begin{document}
\maketitle

\begin{abstract}
We analyze the set of increasingly enumerable additive submonoids of ${\mathbb R}$,
for instance, the set of logarithms of the positive integers with respect to a given base. We call them {\em $\omega$-monoids}. The $\omega$-monoids for which consecutive elements become arbitrarily close are called {\em tempered monoids}. This is, in particular, the case for the set of logarithms.
We show that any $\omega$-monoid is either a scalar multiple of a numerical semigroup or a tempered monoid. We will also show how we can differentiate $\omega$-monoids that are multiples of numerical semigroups from those that are tempered monoids by the size and commensurability of their minimal generating sets.
All the definitions and results are illustrated with examples from music theory.
\end{abstract}


\bigskip

\section{Introduction.}

Music scales are the (ordered) sets of notes that properly combined, together with their octaves, give rise to musical compositions. The exact pitch of each note in the scale is defined by the {\em temperament}, and has been historically defined in mathematical terms in different ways; an important influence is the physics of sound. The references \cite{Barbourbook,HallJosic,Neuwirth,Sethares} explain the mathematics behind the most common temperaments. Ignoring the hierarchy imposed by the physical phenomena, the relationship among the notes in the scale was revisited with the dodecaphonism of Schoenberg, Berg, and Webern. See \cite{Perle} for the mathematical perspective of atonalism. Equal temperament is mandatory in this new order.

Temperament is usually based on the physical phenomenon of the presence of {\em harmonics} or {\em overtones} of a tone when this tone is played by an acoustic instrument. These harmonics are upper tones with lower intensity and their frequencies are positive integer multiples of the frequency of the tone that has been played, which is called the {\em fundamental tone}.
As an example, the harmonics corresponding to $C2$ (where $C4$ is middle $C$ and $C2$ is two octave below $C4$) are represented by increasing order of their frequency as follows.
\begin{center}
  \ifwithlilypond{\lilypondfile[staffsize=16]{whiteharmonics.ly}}{\resizebox{.6\textwidth}{!}{\includegraphics{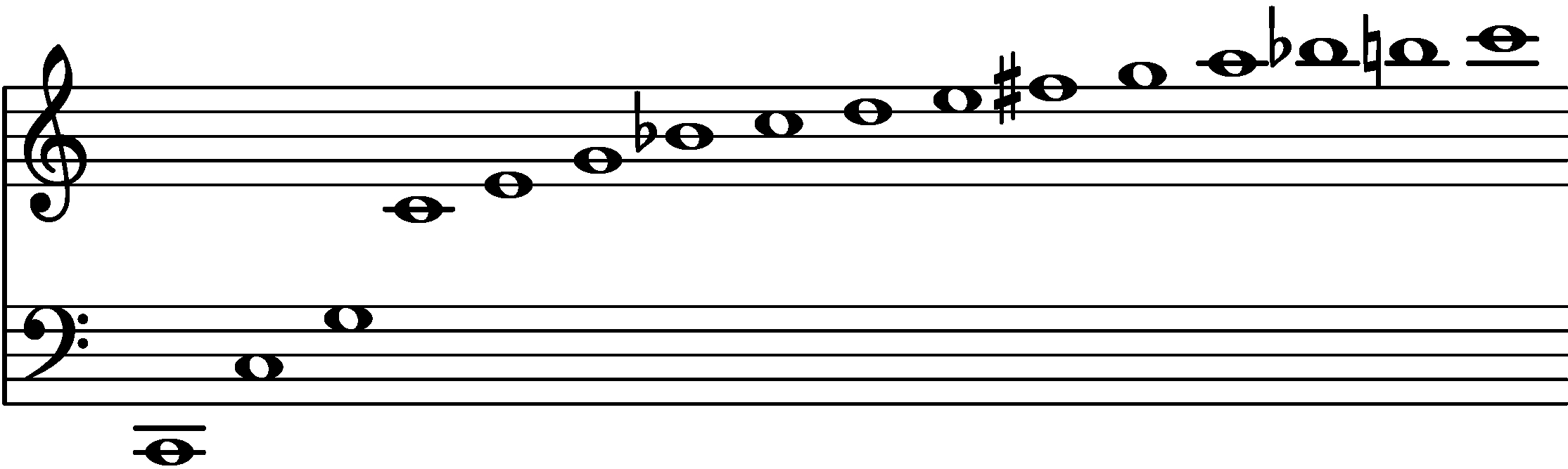}}}
\end{center}
In the definition of temperaments, the first note, known as the {\em fundamental note}, is fixed, and the pitches of the subsequent notes are defined so that the notes or their octave shifts approximate overtones of the fundamental note.

The frequencies of the physical harmonics are integer multiples of the frequency of the fundamental tone. For instance, the octave, which corresponds to the first harmonic, has frequency equal to twice the frequency of the fundamental tone.
On the other hand, the pitch distances between the physical harmonics with respect to the fundamental tone are modeled by the sequence of logarithms in a given base $L_b=\{\log_b(i): i\in{\mathbb N}\},$ where $i$ represents the frequency ratio of each harmonic with the fundamental.
The first step in the sequence, from $\log_b(1)$ to $\log_b(2)$, defines the octave interval. 
Different temperaments fix the rest of the notes of the scale in this first octave interval.
The base $b$ can be thought of as a choice of measurement unit.
For instance, we can take the base $b$ so that $L_b=\{12\log_2(i): i \in {\mathbb N}\}$; see Table~\ref{t:logtab}.
\begin{table}
  \begin{center}
    \resizebox{\textwidth}{!}{
\begin{tabular}{|r|c|c|c|}
\hline
 & \begin{tabular}{l}Relative frequencies\\ w.r.t. fundamental\end{tabular} & \begin{tabular}{l}Perfect pitch distance\\ in semitones\\ w.r.t. fundamental\end{tabular} & \begin{tabular}{l}Approximate \\pitch distance\\ in semitones\\ w.r.t. fundamental\end{tabular} \\\hline
Fundamental & $1$ & $12\log_2(1)$ & $[12\log_2(1)]=0$\\
Octave & $2$ & $12\log_2(2)$ & $[12\log_2(2)]=12$\\
Fifth of the first octave & $3$ & $12\log_2(3)$ & $[12\log_2(3)]=19$\\
Second octave & $4$ & $12\log_2(4)$ & $[12\log_2(4)]=24$\\
Third of the second octave & $5$ & $12\log_2(5)$ & $[12\log_2(5)]=28$\\
\vdots & \vdots & \vdots & \vdots \\
\hline
\end{tabular}}
\end{center}
\caption{The first harmonics in the sequence represented by their relative frequency and their pitch distance with respect to the fundamental tone, as well as their integer approximation.
}\label{t:logtab}
\end{table}
In this way, the fundamental tone is set to $12\log_2(1)=0$ units and the octave is set to $12\log_2(2)=12$ units.
\begin{center}
  \ifwithps{%
\scalebox{.5}{\keyboardI{$0$}{}{}{}{}{}{}{}{}
\keyboardII{}{}{}{}{$12$}}}{\resizebox{.55\textwidth}{!}{\includegraphics{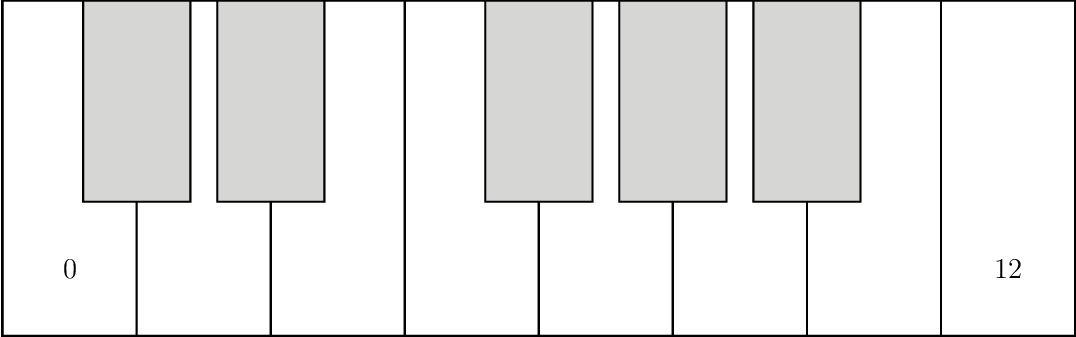}}}
\end{center}
The third note in the harmonic sequence would correspond to $12\log_2(3)=19.01955$ units. Since this is out of the range of the first octave, we reduce it to the first octave by subtracting from it one full octave, that is, subtracting $12$ units from it. Then we obtain a note which is exactly $7.01855/12$ parts of the octave. The first octave over this note will correspond to the third harmonic of the fundamental note.
This note is what is called a {\em perfect fifth}. The perfect fifth is the basis for the so-called Pythagorean tuning, explained later in Example~\ref{ex:pitagoras}.
\begin{center}
  \ifwithps{%
\scalebox{.5}{\keyboardI{$0$}{}{}{}{}{}{}{$7.01855$}{}
\keyboardII{}{}{}{}{$12$}}}{\resizebox{.55\textwidth}{!}{\includegraphics{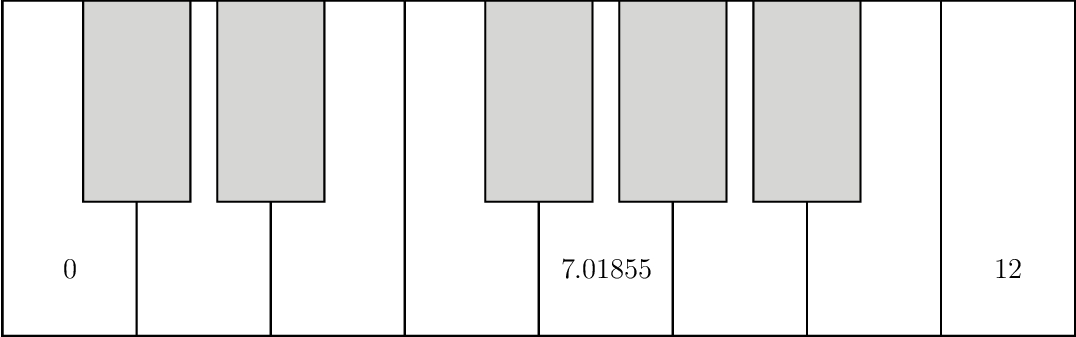}}}
\end{center}
The fourth harmonic is at $12\log_2(4)=24$ units from the fundamental, which corresponds to two exact octaves. The fifth harmonic is at $12\log_2(5)=27.86314$ units from the fundamental. We can subtract from it two octaves (equal to $24$ units) thus obtaining another note at $3.86314$ units from the fundamental. This is the {\em perfect third}.
\begin{center}
  \ifwithps{
\scalebox{.5}{\keyboardI{$0$}{}{}{}{$3.86314$}{}{}{$7.01855$}{}
\keyboardII{}{}{}{}{$12$}}}{\resizebox{.55\textwidth}{!}{\includegraphics{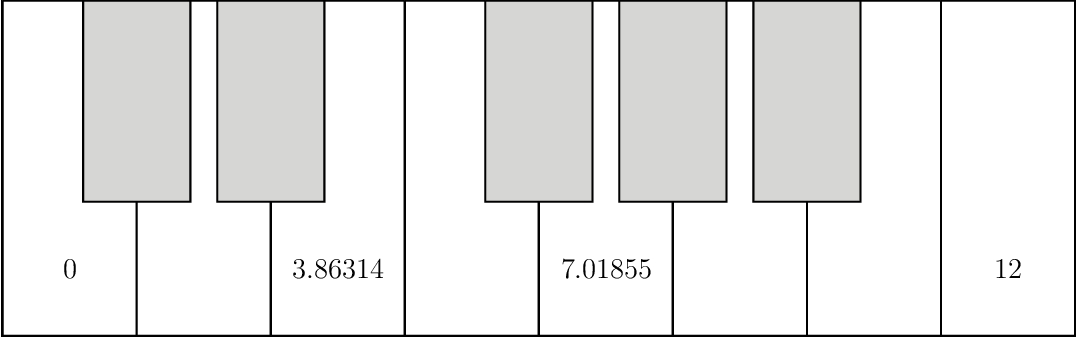}}}
\end{center}
We could continue tuning the rest of the scale in this way.
At the end, any note would be tuned as $12$ times the fractional part of $\log_2(i)$ for some $i>0$.
This perfect tuning, apart from being very difficult to obtain by hand (indeed, by ear), makes it impossible to modulate from one scale based on a fundamental note to another scaled based on another fundamental tone.

Alternatively, {\em equal temperaments} are defined by the condition that the intervals between consecutive notes are constant and, historically, the scale has been divided into $12$ equal parts or {\em semitones}. This way, taking as unit a semitone, the pitch tuning is as easy as in this diagram.
\begin{center}
\ifwithps{
\scalebox{.5}{\keyboardI{$0$}{$1$}{$2$}{$3$}{$4$}{$5$}{$6$}{$7$}
\keyboardII{$8$}{$9$}{$10$}{$11$}{$12$}}}{\resizebox{.55\textwidth}{!}{\includegraphics{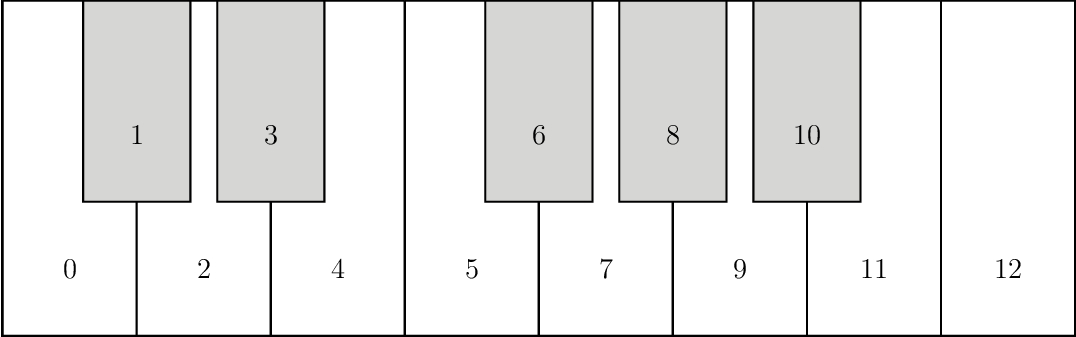}}}
\end{center}
Apart from the $12$-equal temperament, other interesting equal temperaments have been used. For instance, Jeff Harrington has several compositions in the $19$-equal temperament \cite{Harrington}, and Fabio Costa has compositions in the $31$-equal temperament \cite{Costa}, where the octave is divided respectively into $19$ and $31$ equal parts.
  
With the $12$-equal temperament, the set of distances of the harmonics from the fundamental note can be approximated as in the last column of Table~\ref{t:logtab}.
The resulting set of the approximate integer numbers of semitones in the intervals from the fundamental to each of the harmonics
\begin{figure}
  \begin{center}
    \ifwithlilypond{\lilypondfile[staffsize=16]{whiteharmonicssemitonecount.ly}}{\resizebox{.6\textwidth}{!}{\includegraphics{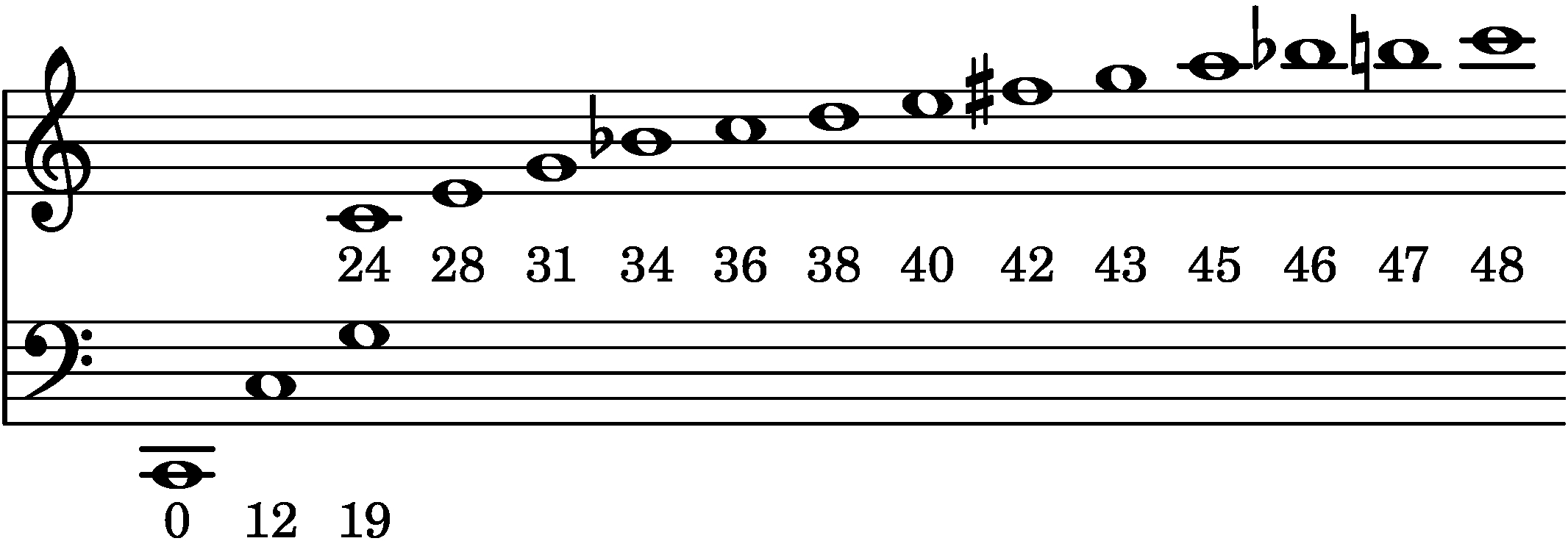}}}
    \end{center}
\caption{Approximate integer numbers of semitones in the intervals from the fundamental to each of the harmonics.}\label{f:whsc}
\end{figure}
  is $$H = \{0, 12, 19, 24, 28, 31, 34, 36, 38, 40, 42, 43, 45, 46, 47, 48\}\cup\{i\in{\mathbb N}: i>48\}.$$ See Figure~\ref{f:whsc}.

  Both the sets $L_b$ (for a fixed $b$) and $H$ are different examples of increasingly enumerable submonoids of real numbers.  
  The set $L_b$ is closed under addition since $\log_b(i)+\log_b(j)=\log_b(ij)$.
  The set $H$ is closed under addition as can be checked by hand.
  Increasingly enumerable submonoids of the real numbers are the main topic of the present article and they will be called {\em $\omega$-monoids} throughout. They are presented in Section~\ref{s:om}.

In Section~\ref{s:ns} numerical semigroups are described. They are a major example of $\omega$-monoids. The set $H$ is an example of numerical semigroup.
On the other hand, tempered monoids are another important example of $\omega$-monoids, and they are described in Section~\ref{s:tm}.
The set $L_b$ is an example of a tempered monoid.

In Section~\ref{s:main} we prove our main result. That is, we prove that any $\omega$-monoid is either a scalar multiple of a numerical semigroup or a tempered monoid. Thus the examples in Sections \ref{s:ns} and \ref{s:tm} essentially constitute the whole set of $\omega$-monoids.

In Section~\ref{s:gen} we see that $\omega$-monoids have a unique minimal generating set. We will show how we can differentiate $\omega$-monoids that are multiples of numerical semigroups from those that are tempered monoids by their minimal generating sets.

Although the new results presented here are of a strictly mathematical nature, all the definitions and results are motivated and illustrated with examples borrowed from music theory.

The connection between algebraic structures and music theory has been widely studied. Apart from \cite{Bras:MM}, to which the present article is a natural continuation, many other references can be cited. For instance, \cite{Andreatta} gives a theoretical approach to music theory from the perspective of group theory; in \cite{Amiot} there is a connection between rhythmic canons and Galois theory; and in \cite{Crans} there is a musical interpretation of the dihedral group of order $24$. From a combinatorial point of view, \cite{Clampitt} presents a connection between musical scale theory and word theory, while \cite{Lewin, Noll} analyze music intervals in terms of mathematical transformations.

\section{$\omega$-monoids.}
\label{s:om}

We denote by {\em$\omega$-monoids} the submonoids of $({\mathbb R},+)$
that are increasingly enumerable or, equivalently, the submonoids of $({\mathbb R},+)$ of order type $\omega$.

That is, the set $\Omega=\{a_0,a_1,\dots\}$ with $a_i<a_{i+1}$ for all nonnegative integers $i$ is an $\omega$-monoid if and only if
\begin{enumerate}
\item[(i)] $a_0=0$, \item[(ii)] $a_i+a_j\in \Omega$ for any pair of positive integers $i,j$.
  \end{enumerate}

There is important literature to contextualize $\omega$-monoids.
Totally ordered commutative monoids were studied by Clifford in \cite{Clifford}. As one of the reviewers of this article pointed out, infinitely generated submonoids of $({\mathbb R}^+,+)$ were crucial to finding very important examples in commutative algebra. For instance, Grams \cite{Grams} uses the submonoid of $({\mathbb R}^+,+)$ generated by the set $\{ \frac{1}{2^np_n} : n \in {\mathbb N}\}$, where $p_1, p_2,\dots$ is an increasing enumeration of the odd primes, to find the first example of an atomic integral domain failing to satisfy the ascending chain condition on principal ideals. On the other hand, in their highly cited paper \cite{Anderson}, Anderson et al. use different submonoids of $({\mathbb R}^+, +)$ to construct integral domains with certain specific factorization properties.

The connection between $\omega$-monoids and music theory was explored in \cite{Bras:MM}. The connection of other submonoids of $({\mathbb R}^+,+)$ with fuzzy logic was studied in \cite{Horcik}. The papers \cite{Gotti,GottiGotti} deal with submonoids of $({\mathbb Q},+)$ and connect them with nonunique factorization theory.

\section{The example of numerical semigroups.}\label{s:ns}
Numerical semigroups are examples of $\omega$-monoids. Let ${\mathbb N}_0={\mathbb N}\cup\{0\}$. A {\em numerical semigroup} $S$ is a subset $S$ of ${\mathbb N}_0$ satisfying
\begin{enumerate}
\item[(i)] $0\in S$, \item[(ii)] $a+b\in S$ whenever $a,b\in S$, \item[(iii)] ${\mathbb N}_0\setminus S$ is finite.
  \end{enumerate}

\begin{example}
  \label{e:hermite}
The set $$S=\{0,4,5,8,9,10\}\cup\{i\in{\mathbb N}: i\geq 12\}$$ is a numerical semigroup. 
\end{example}

\begin{example}
\label{ex:H}
The well-tempered harmonic numerical semigroup is $$H = \{0, 12, 19, 24, 28, 31, 34, 36, 38, 40, 42, 43, 45, 46, 47, \dots\}.$$
Its elements are the numbers of semitones in the intervals between a fundamental note and each of its harmonics, or overtones, when the harmonics are mapped to the equal temperament of $12$ notes per octave
(see Figure~\ref{f:whsc}).
It is easy to check that $H$ is a numerical semigroup.

\end{example}

The {\em gaps} of a numerical semigroup $S$ are the elements in
${\mathbb N}_0\setminus S$. The {\em genus} of $S$, denoted $g(S)$, is the
number of gaps of $S$. The {\em multiplicity} of $S$, denoted $m(S)$, 
is the smallest nonzero nongap of $S$. 
The gaps of the semigroup $S$ in Example~\ref{e:hermite} are
$1,2,3,6,7,11$ and the genus and multiplicity are, respectively, $g(S)=6$,
$m(S)=4$. 
The multiplicity of the semigroup $H$ in Example~\ref{ex:H} is $m(H)=12$ and its genus is $g(H)=33$.

The next result is well known in the theory of numerical semigroups and its proof is straightforward.

\begin{lemma}\label{l:semigroupgcd}
  If a subset $S$ of the nonnegative integers satisfies
  (i) $0\in S$, (ii) $a+b\in S$ whenever $a,b\in S$,
  then the set $S'=\frac{S}{\gcd(S)}$ satisfies
  (i) $0\in S'$, (ii) $a+b\in S'$ whenever $a,b\in S'$,
  (iii) ${\mathbb N}\setminus S'$ is finite.
  That is, $S'$ is a numerical semigroup.
\end{lemma}

There are many open problems related to numerical semigroups. For instance,
while it is known that the sequence of the number of numerical semigroups of each given genus asymptotically behaves like the Fibonacci numbers \cite{Br:fibonacci,Zhai}, it remains open to prove that the number of numerical semigroups of genus $g$ is at most the number of numerical semigroups of genus $g+1$ for any genus \cite{Br:fibonacci}. This conjecture was announced by the author of this article in 2007 at the Thematic Seminar ``Algebraic Geometry, Coding and Computing'' held in Segovia, Spain \cite{ngu}. See \cite{Kaplan} for a nice survey on related results.

\section{The example of tempered monoids.}\label{s:tm}
 
In the context of the mathematics of music \cite{Bras:MM}, a tempered monoid is defined as the set of elements in an increasing sequence $T=\tau_0,\tau_1,\tau_2,\dots$ of ${\mathbb R}_0^+:=\{r\in{\mathbb R}: r\geq 0\}$
such that
\begin{enumerate}
\item[(i)] $\tau_0=0$,
\item[(ii)]
  $\tau_i+\tau_j\in T$ for any pair of positive integers $i,j$,
\item[(iii)] $\lim_{n\to \infty}(\tau_{n+1}-\tau_n)=0$.
  \end{enumerate}
Obviously, a tempered monoid is an $\omega$-monoid. 

If $\tau_1=1$ then we say
that the tempered monoid is {\em normalized}.
For a normalized tempered monoid $T$, we define its {\em $i$th period}, denoted $\pi_i(T)$, to be the set of elements in $T$ that are at least as large as $i$ and strictly smaller than $i+1$.
The {\em granularity} of $T$ is the cardinality of its first period.

\begin{example}
  \label{e:quarterstempered}
  The following set is a tempered monoid:
$$Q=\{0\}\cup\left\{n+ \frac{k}{2^{n+1}} :n\in {\mathbb N} \mbox{ and } 0\leq k\leq 2^{n+1}-1\right\}.$$
Its granularity is $4$. Its first period is $\{1, 1+\frac{1}{4},1+\frac{1}{2},1+\frac{3}{4}\}$.
Its second period is $\{2,2+\frac{1}{8},2+\frac{1}{4},2+\frac{3}{8},2+\frac{1}{2},2+\frac{5}{8},2+\frac{3}{4},2+\frac{7}{8}\}$,
and so on.
\end{example}

\begin{example}
\label{e:deustempered}
  The following set is a tempered monoid:
$$D=\{0\}\cup\left\{n+ \frac{k}{10^{n}} :n\in {\mathbb N} \mbox{ and } 0\leq k\leq 10^{n}-1\right\}.$$
  Its granularity is $10$ and its $i$th period has cardinality $10^i$.
\end{example}

\begin{example}[The logarithmic monoid]
  \label{ex:logtempered}
  The elements in the sequence $L=(\log_2(i+1))_{i\in {\mathbb N}_0}$ constitute another $\omega$-monoid.

A tempered monoid $T=\tau_0, \tau_1, \tau_2, \dots$ is said to be {\em product-compatible} if $\tau_{ab-1}=\tau_{a-1}+\tau_{b-1}$ for any $a,b\in{\mathbb N}$.
It is proved in \cite{Bras:MM} that the unique product-compatible tempered monoid is the logarithmic monoid.
\end{example}

\begin{example}
  [The golden fractal monoid]
\label{ex:fractempered}
We can divide a segment in a fractal way as follows. First
we halve it. Then we halve each half and so on, indefinitely.
See Figure~\ref{f:fractaldivisionofaninterval}.
The same idea can be applied by dividing the interval into two parts in
a given proportion, not necessarily into two equal parts. Next, divide each
of the parts following the same proportions as in the first cut. Divide
again each of the parts in the same proportions and so on. We obtain an
apparently chaotic but strictly fractal partition.
See Figure~\ref{f:fractaldivisionofanintervalnonbisectional}.

\begin{figure}
\centering
\begin{minipage}{.48\textwidth}
  \centering
\ifwithps{
  \psset{unit=7cm}
\resizebox{\textwidth}{!}{
\begin{tabular}{c}
\begin{pspicture}(0,0)(1,.125)
\psaxes[labels=none,ticks=none,linewidth=.005](0,0.1)(1,0.1)
\psaxes[labels=none,ticks=none,linewidth=.005](0,0.075)(0,0.125)
\put(0,.05){\makebox(0,0){$0$}}
\psaxes[labels=none,ticks=none,linewidth=.005](1,0.075)(1,0.125)
\put(1,.05){\makebox(0,0){$1$}}
\end{pspicture}
\\
\begin{pspicture}(0,0)(1,.125)
\psaxes[labels=none,ticks=none,linewidth=.005](0,0.1)(1,0.1)
\psaxes[labels=none,ticks=none,linewidth=.005](0,0.075)(0,0.125)
\put(0,.05){\makebox(0,0){$0$}}
\psaxes[labels=none,ticks=none,linewidth=.005](1,0.075)(1,0.125)
\put(1,.05){\makebox(0,0){$1$}}
\psaxes[labels=none,ticks=none,linewidth=.005](.5,0.075)(.5,0.125)
\put(.5,.05){\makebox(0,0){$.5$}}
\end{pspicture}
\\
\begin{pspicture}(0,0)(1,.125)
\psaxes[labels=none,ticks=none,linewidth=.005](0,0.1)(1,0.1)
\psaxes[labels=none,ticks=none,linewidth=.005](0,0.075)(0,0.125)
\put(0,.05){\makebox(0,0){$0$}}
\psaxes[labels=none,ticks=none,linewidth=.005](1,0.075)(1,0.125)
\put(1,.05){\makebox(0,0){$1$}}
\psaxes[labels=none,ticks=none,linewidth=.005](.5,0.075)(.5,0.125)
\put(.5,.05){\makebox(0,0){$.5$}}
\psaxes[labels=none,ticks=none,linewidth=.005](.25,0.075)(.25,0.125)
\put(.25,.05){\makebox(0,0){$.25$}}
\psaxes[labels=none,ticks=none,linewidth=.005](.75,0.075)(.75,0.125)
\put(.75,.05){\makebox(0,0){$.75$}}
\end{pspicture}
\\
\begin{pspicture}(0,0)(1,.125)
\psaxes[labels=none,ticks=none,linewidth=.005](0,0.1)(1,0.1)
\psaxes[labels=none,ticks=none,linewidth=.005](0,0.075)(0,0.125)
\put(0,.05){\makebox(0,0){$0$}}
\psaxes[labels=none,ticks=none,linewidth=.005](1,0.075)(1,0.125)
\put(1,.05){\makebox(0,0){$1$}}
\psaxes[labels=none,ticks=none,linewidth=.005](.5,0.075)(.5,0.125)
\put(.5,.05){\makebox(0,0){$.5$}}
\psaxes[labels=none,ticks=none,linewidth=.005](.25,0.075)(.25,0.125)
\put(.25,.05){\makebox(0,0){$.25$}}
\psaxes[labels=none,ticks=none,linewidth=.005](.75,0.075)(.75,0.125)
\put(.75,.05){\makebox(0,0){$.75$}}
\psaxes[labels=none,ticks=none,linewidth=.005](.125,0.075)(.125,0.125)
\put(.125,.05){\makebox(0,0){$.125$}}
\psaxes[labels=none,ticks=none,linewidth=.005](.375,0.075)(.375,0.125)
\put(.375,.05){\makebox(0,0){$.375$}}
\psaxes[labels=none,ticks=none,linewidth=.005](.625,0.075)(.625,0.125)
\put(.625,.05){\makebox(0,0){$.625$}}
\psaxes[labels=none,ticks=none,linewidth=.005](.875,0.075)(.875,0.125)
\put(.875,.05){\makebox(0,0){$.875$}}
\end{pspicture}
\\
\begin{pspicture}(0,0)(1,.125)
\psaxes[labels=none,ticks=none,linewidth=.005](0,0.1)(1,0.1)
\psaxes[labels=none,ticks=none,linewidth=.005](0,0.075)(0,0.125)
\put(0,.05){\makebox(0,0){$0$}}
\psaxes[labels=none,ticks=none,linewidth=.005](1,0.075)(1,0.125)
\put(1,.05){\makebox(0,0){$1$}}
\psaxes[labels=none,ticks=none,linewidth=.005](.5,0.075)(.5,0.125)
\psaxes[labels=none,ticks=none,linewidth=.005](.25,0.075)(.25,0.125)
\psaxes[labels=none,ticks=none,linewidth=.005](.75,0.075)(.75,0.125)
\psaxes[labels=none,ticks=none,linewidth=.005](.125,0.075)(.125,0.125)
\psaxes[labels=none,ticks=none,linewidth=.005](.375,0.075)(.375,0.125)
\psaxes[labels=none,ticks=none,linewidth=.005](.625,0.075)(.625,0.125)
\psaxes[labels=none,ticks=none,linewidth=.005](.875,0.075)(.875,0.125)
\psaxes[labels=none,ticks=none,linewidth=.005](.0625,0.075)(.0625,0.125)
\psaxes[labels=none,ticks=none,linewidth=.005](.1875,0.075)(.1875,0.125)
\psaxes[labels=none,ticks=none,linewidth=.005](.3125,0.075)(.3125,0.125)
\psaxes[labels=none,ticks=none,linewidth=.005](.4375,0.075)(.4375,0.125)
\psaxes[labels=none,ticks=none,linewidth=.005](.5625,0.075)(.5625,0.125)
\psaxes[labels=none,ticks=none,linewidth=.005](.6875,0.075)(.6875,0.125)
\psaxes[labels=none,ticks=none,linewidth=.005](.8125,0.075)(.8125,0.125)
\psaxes[labels=none,ticks=none,linewidth=.005](.9375,0.075)(.9375,0.125)
\end{pspicture}
\end{tabular}}}{\resizebox{.9\textwidth}{!}{\includegraphics{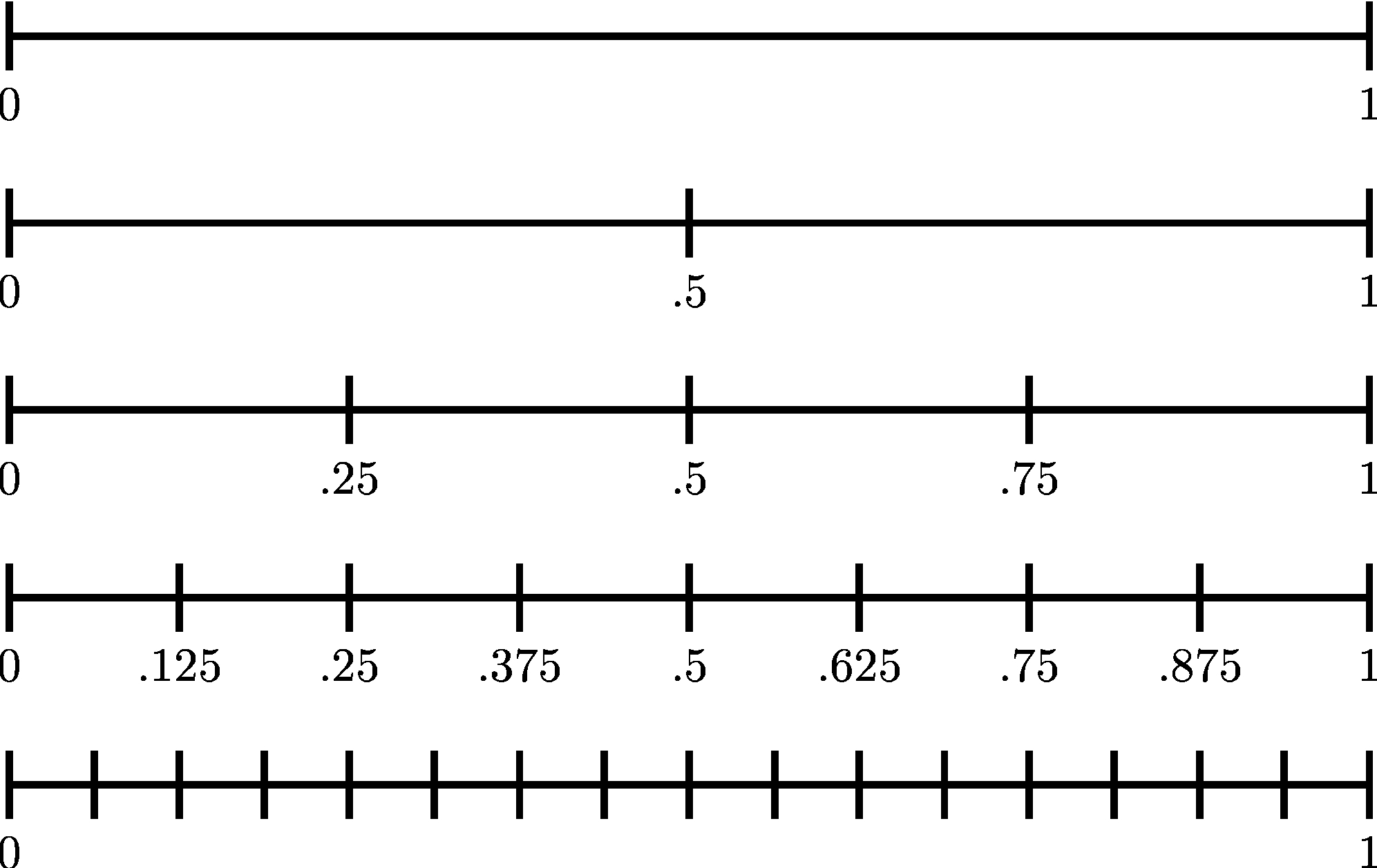}}}
\captionof{figure}{Fractal division of an interval (bisectional).}
\label{f:fractaldivisionofaninterval}
\end{minipage}
\ \ \
\begin{minipage}{.48\textwidth}
  \centering
\ifwithps{
  \psset{unit=7cm}
\resizebox{\textwidth}{!}{
\begin{tabular}{c}
\begin{pspicture}(0,0)(1,.125)
\psaxes[labels=none,ticks=none,linewidth=.005](0,0.1)(1,0.1)
\psaxes[labels=none,ticks=none,linewidth=.005](0,0.075)(0,0.125)
\put(0,.05){\makebox(0,0){$0$}}
\psaxes[labels=none,ticks=none,linewidth=.005](1,0.075)(1,0.125)
\put(1,.05){\makebox(0,0){$1$}}
\end{pspicture}
\\
\begin{pspicture}(0,0)(1,.125)
\psaxes[labels=none,ticks=none,linewidth=.005](0,0.1)(1,0.1)
\psaxes[labels=none,ticks=none,linewidth=.005](0,0.075)(0,0.125)
\put(0,.05){\makebox(0,0){$0$}}
\psaxes[labels=none,ticks=none,linewidth=.005](1,0.075)(1,0.125)
\put(1,.05){\makebox(0,0){$1$}}
\psaxes[labels=none,ticks=none,linewidth=.005](.618,0.075)(.618,0.125)
\put(.618,.05){\makebox(0,0){$p$}}
\end{pspicture}
\\
\begin{pspicture}(0,0)(1,.125)
\psaxes[labels=none,ticks=none,linewidth=.005](0,0.1)(1,0.1)
\psaxes[labels=none,ticks=none,linewidth=.005](0,0.075)(0,0.125)
\put(0,.05){\makebox(0,0){$0$}}
\psaxes[labels=none,ticks=none,linewidth=.005](1,0.075)(1,0.125)
\put(1,.05){\makebox(0,0){$1$}}
\psaxes[labels=none,ticks=none,linewidth=.005](.618,0.075)(.618,0.125)
\put(.618,.05){\makebox(0,0){$p$}}
\psaxes[labels=none,ticks=none,linewidth=.005](.382,0.075)(.382,0.125)
\put(.382,.05){\makebox(0,0){$p^2$}}
\psaxes[labels=none,ticks=none,linewidth=.005](.8541,0.075)(.8541,0.125)
\put(.8541,.05){\makebox(0,0){\scalebox{.8}{$p+(1-p)p$}}}
\end{pspicture}
\\
\begin{pspicture}(0,0)(1,.125)
\psaxes[labels=none,ticks=none,linewidth=.005](0,0.1)(1,0.1)
\psaxes[labels=none,ticks=none,linewidth=.005](0,0.075)(0,0.125)
\put(0,.05){\makebox(0,0){$0$}}
\psaxes[labels=none,ticks=none,linewidth=.005](1,0.075)(1,0.125)
\put(1,.05){\makebox(0,0){$1$}}
\psaxes[labels=none,ticks=none,linewidth=.005](.618,0.075)(.618,0.125)
\psaxes[labels=none,ticks=none,linewidth=.005](.382,0.075)(.382,0.125)
\psaxes[labels=none,ticks=none,linewidth=.005](.8541,0.075)(.8541,0.125)
\psaxes[labels=none,ticks=none,linewidth=.005](.2361,0.075)(.2361,0.125)
\psaxes[labels=none,ticks=none,linewidth=.005](.5279,0.075)(.5279,0.125)
\psaxes[labels=none,ticks=none,linewidth=.005](.7639,0.075)(.7639,0.125)
\psaxes[labels=none,ticks=none,linewidth=.005](.9443,0.075)(.9443,0.125)
\end{pspicture}
\\
\begin{pspicture}(0,0)(1,.125)
\psaxes[labels=none,ticks=none,linewidth=.005](0,0.1)(1,0.1)
\psaxes[labels=none,ticks=none,linewidth=.005](0,0.075)(0,0.125)
\put(0,.05){\makebox(0,0){$0$}}
\psaxes[labels=none,ticks=none,linewidth=.005](1,0.075)(1,0.125)
\put(1,.05){\makebox(0,0){$1$}}
\psaxes[labels=none,ticks=none,linewidth=.005](.618,0.075)(.618,0.125)
\psaxes[labels=none,ticks=none,linewidth=.005](.382,0.075)(.382,0.125)
\psaxes[labels=none,ticks=none,linewidth=.005](.8541,0.075)(.8541,0.125)
\psaxes[labels=none,ticks=none,linewidth=.005](.2361,0.075)(.2361,0.125)
\psaxes[labels=none,ticks=none,linewidth=.005](.5279,0.075)(.5279,0.125)
\psaxes[labels=none,ticks=none,linewidth=.005](.7639,0.075)(.7639,0.125)
\psaxes[labels=none,ticks=none,linewidth=.005](.9443,0.075)(.9443,0.125)
\psaxes[labels=none,ticks=none,linewidth=.005](.1459 ,0.075)(.1459,0.125)
\psaxes[labels=none,ticks=none,linewidth=.005](.3262 ,0.075)(.3262,0.125)
\psaxes[labels=none,ticks=none,linewidth=.005](.4721 ,0.075)(.4721,0.125)
\psaxes[labels=none,ticks=none,linewidth=.005](.5836,0.075)(.5836,0.125)
\psaxes[labels=none,ticks=none,linewidth=.005](.7082,0.075)(.7082,0.125)
\psaxes[labels=none,ticks=none,linewidth=.005](.8197,0.075)(.8197,0.125)
\psaxes[labels=none,ticks=none,linewidth=.005](.9098,0.075)(.9098,0.125)
\psaxes[labels=none,ticks=none,linewidth=.005](.9787,0.075)(.9787,0.125)
\end{pspicture}
\end{tabular}}}{\resizebox{.9\textwidth}{!}{\includegraphics{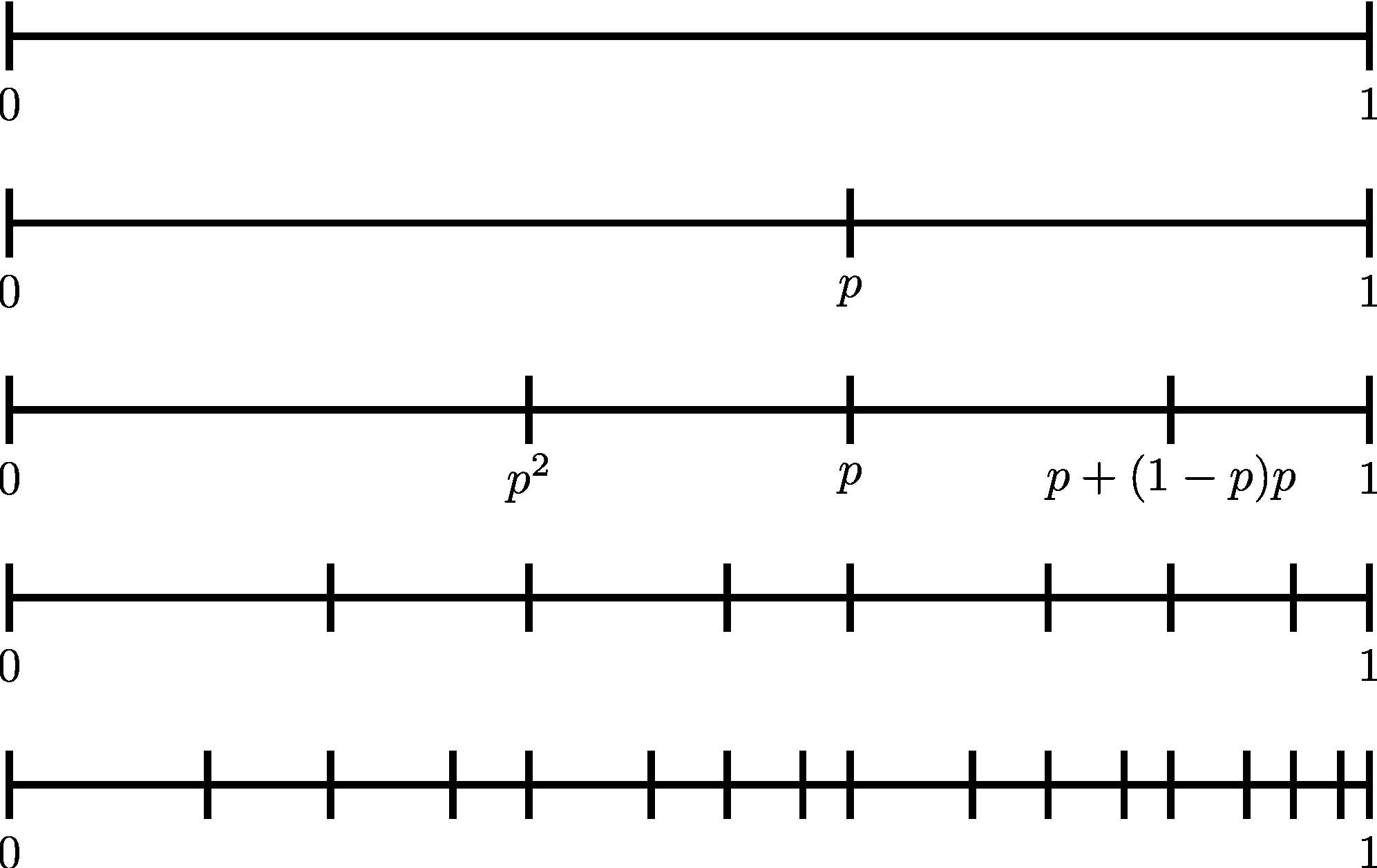}}}
\captionof{figure}{Fractal division of an interval (nonbisectional).}
\label{f:fractaldivisionofanintervalnonbisectional}
\end{minipage}
\end{figure}

A tempered monoid is said to be {\em fractal} if its consecutive periods, seen as divisions of the unit interval, follow the pattern of a fractal division of an interval.
For instance, the tempered monoid in Example~\ref{e:deustempered} is a fractal monoid.

Let $\phi=\frac{1+\sqrt{5}}{2}$ be the golden number.
It is proved in \cite{Bras:MM} that the period $\{1,\phi\}$ generates a fractal tempered monoid $F$ and that the unique nonbisectional normalized fractal monoid of granularity $2$ is exactly $F$. Hence, $F$ is called the {\em golden fractal monoid}.
 
The first (rounded) terms of $F$ are listed below:

\begin{eqnarray*}
  F&=&\{0,
  1,
  1.6180, 2, 2.3820, 2.6180, 2.8541, 3, 3.2361, 3.3820, 3.5279, 3.6180,
  \\&&3.7639, 3.8541, 3.9443, 4, 4.1459, 4.2361, 4.3262, 4.3820, 4.4721, 4.5279, \\&&4.5836, 4.6180, 4.7082, 4.7639, 4.8197, 4.8541, 4.9098, 4.9443, 4.9787, 5, \\&&5.0902, 5.1459, 5.2016, 5.2361, 5.2918, 5.3262, 5.3607, 5.3820, 5.4377, \\&&5.4721, 5.5066, 5.5279, 5.5623, 5.5836, 5.6049, 5.6180, 5.6738, 5.7082, \\&&5.7426, 5.7639, 5.7984, 5.8197, 5.8409, 5.8541, 5.8885, 5.9098, 5.9311, \\&&5.9443, 5.9656, 5.9787, 5.9919, 6, \dots \}.
\end{eqnarray*}

\end{example}

\begin{remark}
  The numerical semigroup $H$ in Example~\ref{ex:H} satisfies
  \begin{itemize}
  \item $H=\lfloor 12 L + 0.6\rfloor$,
  \item $H=\lfloor 12 F\rfloor$,
  \end{itemize}
  where $L$ is the logarithmic monoid (Example~\ref{ex:logtempered})
  and $F$ is the golden fractal monoid (Example~\ref{ex:fractempered}).
\end{remark}

\begin{example}[The Pythagorean monoid]\label{ex:pitagoras}
  The Pythagorean tuning of the musical scale is based on the perfect fifth. This means that the fifth is tuned so that its octave coincides exactly with the third harmonic in the harmonic sequence. For this reason, the frequency ratio of the fifth tone with respect to a fundamental tone is $3:2$. The ``$3$'' corresponds to the third harmonic and the ``$:2$'' corresponds to lowering one octave. Notice that by lowering one octave we have $1<\frac{3}{2}<2$.

  Alternatively, if we focus on pitch differences rather than frequency ratios, then, in the Pythagorean tuning, the pitch difference of the interval between a fundamental tone and its fifth tone is the portion of the octave given by the irrational number $\log_2{\frac{3}{2}}$.
  
The fifth of the fifth can be defined in the same terms and so on,
giving rise to the {\em circle of fifths}. The circle of fifths defines
one by one the notes of the scale tuned in the so-called Pythagorean
tuning. See Figure \ref{figpit1} and Figure \ref{figpit2}.

\begin{figure}
  \begin{center}
    \ifwithps{%
\scalebox{.7}{\keyboardI{$1$}{$\frac{3^7}{2^{11}}$}{$\frac{3^2}{2^3}$}{$\frac{3^9}{2^{14}}$}{$\frac{3^4}{2^6}$}{$\frac{3^{-1}}{2^{-2}}$}{$\frac{3^6}{2^9}$}{$\frac{3}{2}$}
\keyboardII{$\frac{3^8}{2^{12}}$}{$\frac{3^3}{2^4}$}{$\frac{3^{10}}{2^{15}}$}{$\frac{3^5}{2^7}$}{$\frac{3^0}{2^{-1}}$}}}{\resizebox{.55\textwidth}{!}{\includegraphics{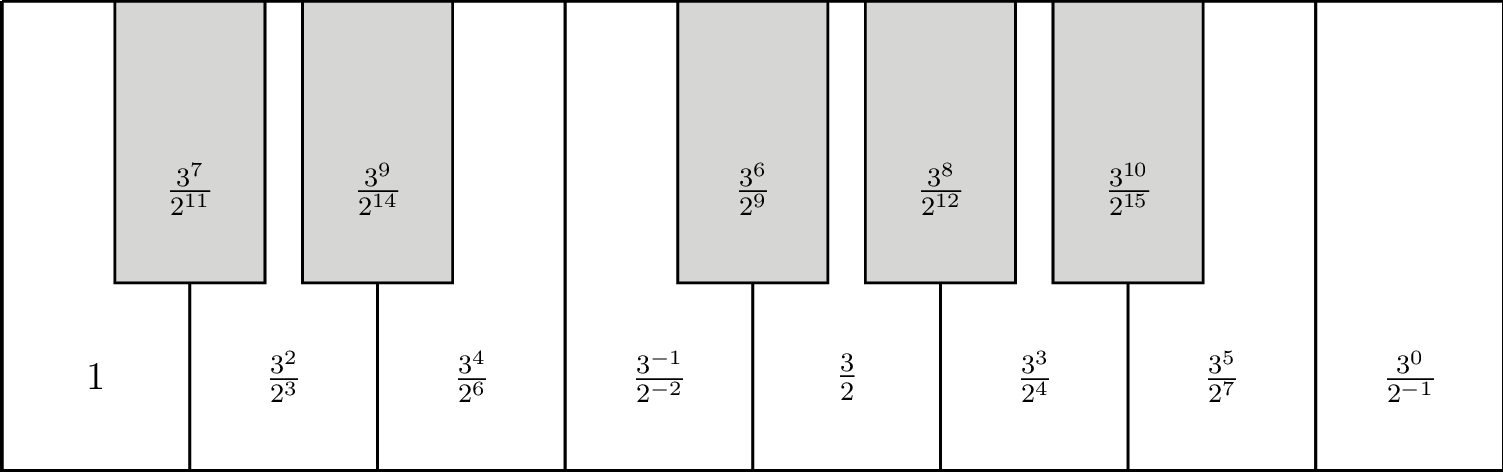}}}
\caption{Frequency ratios with respect to the fundamental (C) in
  Pythagorean tuning.}
\label{figpit1}
\end{center}
\end{figure}

\begin{figure}
  \begin{center}
    \ifwithps{%
\scalebox{.7}{\keyboardI{$0$}{$\log_2\frac{3^7}{2^{11}}$}{$\log_2\frac{3^2}{2^3}$}{$\log_2\frac{3^9}{2^{14}}$}{$\log_2\frac{3^4}{2^6}$}{$\log_2\frac{3^{-1}}{2^{-2}}$}{$\log_2\frac{3^6}{2^9}$}{$\log_2\frac{3}{2}$}
\keyboardII{$\log_2\frac{3^8}{2^{12}}$}{$\log_2\frac{3^3}{2^4}$}{$\log_2\frac{3^{10}}{2^{15}}$}{$\log_2\frac{3^5}{2^7}$}{$1$}}}{\resizebox{.55\textwidth}{!}{\includegraphics{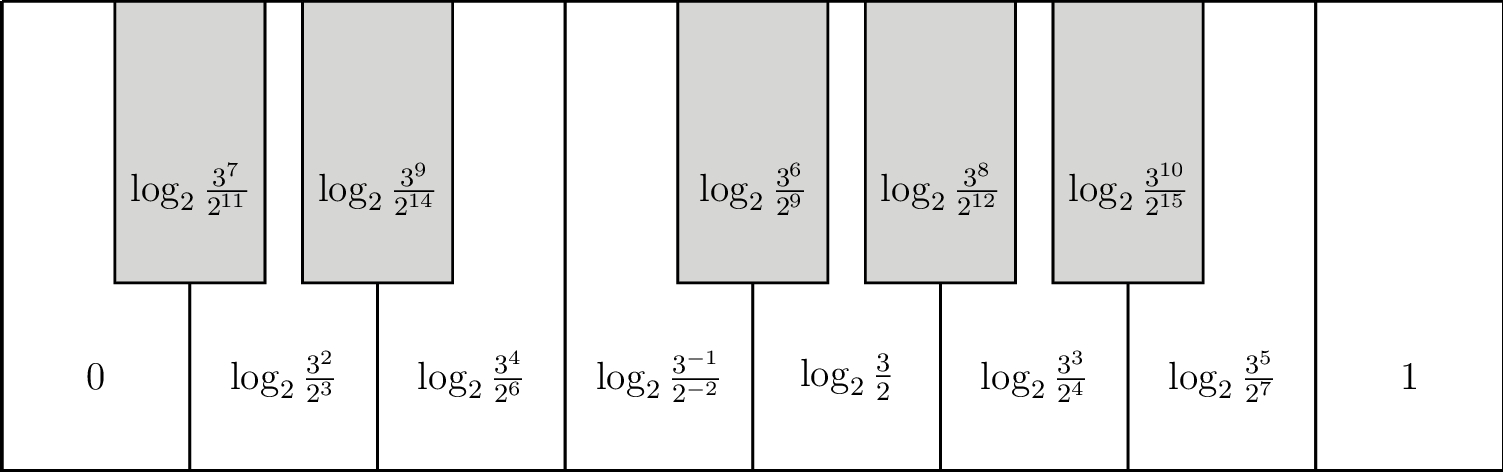}}}
    \caption{Pitch intervals with respect to the fundamental (C) in
      Pythagorean tuning.}
\label{figpit2}
\end{center}
\end{figure}

In the end, the set of pitch differences of all notes with respect to the fundamental note given in terms of portions of the octave lie in $\{\log_2{\frac{3^i}{2^j}}: i,j\in{\mathbb N}_0\}$ (except for the fourth tone, which is usually defined by the frequency ratio $3^{-1}:2^{-1}$).

In this setting, the pitches of the harmonics of a given fundamental tone are approximated by the set $P=\{i+j\log_2{3}:i,j\in{\mathbb N}_0\}$. The set $P$ is a tempered monoid which we will call the {\rm Pythagorean monoid}.
  \end{example}

\section{$\omega$-monoids, tempered monoids, and numerical semigroups.}\label{s:main}

We now state and prove our main theorem relating $\omega$-monoids, tempered monoids, and numerical semigroups.
\begin{theorem}
\label{t:main}
$\omega$-monoids are scalar multiples of either numerical semigroups or normalized tempered monoids. 
\end{theorem}
It is obvious that a tempered monoid
with smallest nonzero element equal to $a_1$ is 
$a_1$ times a normalized tempered monoid.
The rest of the section is devoted to proving that 
if an $\omega$-monoid is not a tempered monoid then it is a scalar multiple of a numerical semigroup.

In order to deal with the mathematical properties of an 
$\omega$-monoid, we will analyze its normalized counterpart and examine the behavior of 
the fractional part of all its elements. In this way, we define the {\em footprint} of an $\omega$-monoid $\Omega$ to be
$$\Delta=\left\{\frac{a}{a_1}-\left\lfloor\frac{a}{a_1}\right\rfloor: a\in \Omega\right\},$$
where $a_1$ is the smallest nonzero element of $\Omega$. 
Obviously, $\Delta$ is a subset of the unit interval $[0,1)$. 

We can think of the positive real line as an infinite {\it slinky}
with diameter equal to one. Then each point in the slinky projects
into a point of the circle of diameter one. The set of the projection
of all points in the slinky corresponding to a normalized $\omega$-monoid would then be its footprint. (This illustrative explanation was suggested by one of the anonymous reviewers.)

Obviously, $\Delta$ is a subset of the unit interval $[0,1)$ closed
under addition mod $1$. In particular, we have the following.

\begin{enumerate}
\item If $\delta_1,\delta_2\in \Delta$ and $\delta_1+\delta_2<1$, then $\delta_1+\delta_2\in\Delta$.
\item If $1-\delta_1,1-\delta_2\in \Delta$, and $1-(\delta_1+\delta_2)\geq 0$, then $1-(\delta_1+\delta_2)\in\Delta$.
\item If $\delta_1,1-\delta_2\in \Delta$ then,
  \begin{itemize}
  \item if $\delta_1\geq \delta_2$, then $\delta_1-\delta_2\in\Delta$;
  \item if $\delta_1<\delta_2$, then $1-\delta_2+\delta_1\in\Delta$.
  \end{itemize}
\end{enumerate}

\begin{lemma}\label{l:pretheorem}
If an $\omega$-monoid is not a tempered monoid, then it is a scalar multiple of a numerical semigroup.
\end{lemma}

\begin{proof}
We can assume, without loss of generality, that $\Omega$ is normalized.

Let $\Delta$ be the footprint of $\Omega$. If $\Delta\setminus\{0\}=\emptyset$ then we are done since, in this case, $\Omega={\mathbb N}_0$.
Hence we assume that $\Delta\setminus\{0\}\neq\emptyset$.

If $\Omega$ is not a tempered monoid, then there exists $\varepsilon>0$ such that, for any $n$, there exists $N>n$ such that $a_{N+1}-a_N>\varepsilon$. Necessarily, $\varepsilon<1$.

Let us first prove that $\Delta\cap(0,\varepsilon)=\emptyset$.
Indeed, if $\delta\in\Delta\cap(0,\varepsilon)$, then, by
the additive closure of $\Delta\mod{1}$,
the elements $\delta,2\delta,\dots,\lfloor\frac{1}{\delta}\rfloor  \delta$ belong to $\Delta$.
  Hence, there exist integers $q_1,q_2,\dots,q_{\lfloor\frac{1}{\delta}\rfloor}$ such that $q_1+\delta\in \Omega,q_2+2\delta\in \Omega,\dots,q_{\lfloor\frac{1}{\delta}\rfloor}+{\lfloor\frac{1}{\delta}\rfloor}\delta\in \Omega$. Let $q$ be the maximum among $q_1,q_2,\dots,q_{\lfloor\frac{1}{\delta}\rfloor}$.
  Then $q\in \Omega$, $q+\delta\in \Omega$, $q+2\delta\in \Omega$, $\dots$, $q+{\lfloor\frac{1}{\delta}\rfloor}\delta\in \Omega$, and the sequence
  $$\begin{array}{rrrcr}
    q, & q+\delta,&q+2\delta,&\dots,&q+{\lfloor\frac{1}{\delta}\rfloor}\delta,\\
    q+1, & q+1+\delta,&q+1+2\delta,&\dots,&q+1+{\lfloor\frac{1}{\delta}\rfloor}\delta,\\
    q+2, & q+2+\delta,&q+2+2\delta,&\dots,&q+2+{\lfloor\frac{1}{\delta}\rfloor}\delta,\\
    \dots &  &&&\\
  \end{array}$$
    is an unbounded subset of $\Omega$ whose consecutive elements are all at most $\delta<\varepsilon$ apart. This contradicts the choice of $\varepsilon$.
It can be similarly proven that $\Delta\cap(1-\varepsilon,1)=\emptyset$.
Indeed, if $1-\delta\in\Delta\cap(1-\varepsilon,1)$, then,
by the additive closure of $\Delta\mod{1}$, 
the elements $1-\delta,1-2\delta,\dots,1-\lfloor\frac{1}{\delta}\rfloor  \delta$ belong to $\Delta$.
  Hence, there exist integers $q_1,q_2,\dots,q_{\lfloor\frac{1}{\delta}\rfloor}$ such that $q_1+1-\delta\in \Omega,q_2+1-2\delta\in \Omega,\dots,q_{\lfloor\frac{1}{\delta}\rfloor}+1-{\lfloor\frac{1}{\delta}\rfloor}\delta\in \Omega$. Let $q$ be the maximum among $q_1,q_2,\dots,q_{\lfloor\frac{1}{\delta}\rfloor}$.
  Then $q\in \Omega$, $q+1-{\lfloor\frac{1}{\delta}\rfloor}\delta\in \Omega$, $\dots$, $q+1-2\delta\in \Omega$, $q+1-\delta\in \Omega$
and the sequence
$$q, 
q+1-{\left\lfloor\frac{1}{\delta}\right\rfloor}\delta, 
\dots,
q+1-\delta,
    q+1, 
q+2-{\left\lfloor\frac{1}{\delta}\right\rfloor}\delta,
\dots
$$
  is an unbounded subset of $\Omega$ whose consecutive elements are all at most $\delta<\varepsilon$ apart. Again this contradicts the choice of $\varepsilon$.

Let $\lambda=\inf(\Delta\setminus\{0\})$ and
$\mu=\inf\{1-\delta:\delta\in\Delta\setminus\{0\}\}$. We see that
$\lambda=\mu$. By the previous arguments,
$\lambda\geq\varepsilon$ and
$\mu\geq\varepsilon$. If $\lambda>\mu$, take
$\lambda_0\in\Delta$ such that $\lambda\leq
\lambda_0<\lambda+\varepsilon$ and
$\mu_0\in\{1-\delta:\delta\in\Delta\setminus\{0\}\}$ such that
$\mu\leq \mu_0<\lambda$. Then $0<\lambda_0-\mu_0$ and 
$\lambda_0-\mu_0<\lambda+\varepsilon-\mu\leq \lambda$. But, by the additive closure of $\Delta\mod{1}$, we have
$\lambda_0-\mu_0\in\Delta$.
This contradicts $\lambda$ being the infimum of $\Delta\setminus\{0\}$.
On the other hand, if $\lambda<\mu$, take
$\lambda_0\in\Delta$ such that $\lambda\leq \lambda_0<\mu$ and
$\mu_0\in\{1-\delta:\delta\in\Delta\setminus\{0\}\}$ such that
$\mu\leq \mu_0<\mu+\varepsilon$. Then
$1-\mu_0+\lambda_0<1-\mu+\mu=1$ and also
$1-\mu_0+\lambda_0>1-\mu-\varepsilon+\lambda\geq 1-\mu>0$.
Now, by the additive closure of $\Delta\mod{1}$, 
it follows that $1-\mu_0+\lambda_0\in\Delta$. Hence,
$\mu_0-\lambda_0\in\{1-\delta:\delta\in\Delta\setminus\{0\}\}$, but 
$\mu_0-\lambda_0<\mu+\varepsilon-\lambda\leq \mu$, a contradiction with $\mu$ being the 
infimum of $\{1-\delta:\delta\in\Delta\setminus\{0\}\}$.
We thus conclude that $\lambda=\mu$.

Let us see now that $\lambda$ and $1-\mu$ belong to
$\Delta\setminus\{0\}$. Suppose first that $\lambda$ does not belong
to $\Delta\setminus\{0\}$. Then it must be an accumulation point of
$\Delta\setminus\{0\}$ and so there exists $\lambda_0\in \Delta\setminus\{0\}$ with $\lambda<\lambda_0<2\lambda$. There also
exists $\mu_0\in\{1-\delta:\delta\in\Delta\setminus\{0\}\}$ such that
$\mu\leq \mu_0 <\lambda_0$. By 
the additive closure of $\Delta\mod{1}$,
it follows that
$\lambda_0-\mu_0\in\Delta$, but $0<\lambda_0-\mu_0<\lambda$, a
contradiction. Suppose now that $1-\mu$ does not belong to
$\Delta\setminus\{0\}$. Then $\mu$ must be an accumulation point of
$\{1-\delta:\delta\in\Delta\setminus\{0\}\}$ and so there exists
$\mu_0\in \{1-\delta:\delta\in\Delta\setminus\{0\}\}$ with
$\mu<\mu_0<2\mu$. There also exists $\lambda_0\in\Delta\setminus\{0\}$
such that $\lambda\leq \lambda_0 <\mu_0$. By 
the additive closure of $\Delta\mod{1}$, it follows that
$1+\lambda_0-\mu_0\in\Delta$, but $1>1+\lambda_0-\mu_0>1-\mu$, a contradiction.

Now, since $\lambda,1-\mu\in\Delta\setminus\{0\}$, we have
$\{\lambda,2\lambda,3\lambda,\dots\}\cap(0,1)\subseteq\Delta$ and
$\{1-\mu,1-2\mu,1-3\mu,\dots\}\cap(0,1)\subseteq\Delta$.
On the other hand,
$\Delta\setminus\{\lambda,2\lambda,3\lambda,\dots\}\cap(0,1)=\emptyset$
because, otherwise, if there exists $\delta\in\Delta$ such that
$k\lambda<\delta<(k+1)\lambda$ for some integer $k$, then
$\delta-k\mu\in\Delta$ by 
the additive closure of $\Delta\mod{1}$. Now, $0<\delta-k\mu<\lambda$, which contradicts that $\lambda$ is the minimum in $\Delta\setminus\{0\}$. Furthermore, $1=k\lambda$ for some integer $k$ since, otherwise, the element $1-\mu\in\Delta$ would be in $\Delta\setminus\{\lambda,2\lambda,3\lambda,\dots\}\cap(0,1)$, which we proved to be empty.
Consequently, $$\Delta=\left\{0,\lambda,2\lambda,\dots,\left(\frac{1}{\lambda}-1\right)\lambda\right\}.$$ 
Then $\Omega\subseteq (\Delta+{\mathbb N}_0)$ and so $\frac{\Omega}{\lambda}\subseteq {\mathbb N}_0$. Now, by Lemma~\ref{l:semigroupgcd}, $\frac{\Omega}{\lambda\gcd\left(\frac{\Omega}{\lambda}\right)}$ is a numerical semigroup.

Hence, $\Omega$ is $\lambda\gcd\left(\frac{\Omega}{\lambda}\right)$ times the numerical semigroup $\frac{\Omega}{\lambda\gcd\left(\frac{\Omega}{\lambda}\right)}$.
\end{proof}

\section{The minimal generating set of an $\omega$-monoid.}\label{s:gen}

For an $\omega$-monoid we can consider minimal generating sets, so that finite sums of their elements generate the monoid. 
If $\Omega$ is an $\omega$-monoid, we denote by $\Omega^*$ the set of its nonzero elements. 
It is easy to check that the set $\Omega\setminus (\{0\}\cup(\Omega^*+\Omega^*))$ generates $\Omega$ and that it is minimal.
Observe that, as opposed to other algebraic structures, for $\omega$-monoids the minimal generating set is unique.
The unique minimal generating set of $\Omega$ will be denoted ${\mathcal G}(\Omega)$.

\begin{remark}\label{r:sfinite}The minimal generating set of a numerical semigroup is always finite. Indeed, all the generators except the multiplicity $m$ are included in the
set
  formed by the minimum element of the semigroup belonging to each class modulo $m$ of the semigroup, and which is obviously finite. 
\end{remark}
  
\begin{example}
  For the numerical semigroup
  $S=\{0,4,5,8,9,10\}\cup\{i\in{\mathbb N}: i\geq 12\}$
  in Example~\ref{e:hermite}, ${\mathcal G}(S)=\{4,5\}$.
\end{example}\begin{example}
  For the numerical semigroup
  $$H = \{0, 12, 19, 24, 28, 31, 34, 36, 38, 40, 42, 43, 45, 46, 47, \dots\}$$
  in Example~\ref{ex:H}, ${\mathcal G}(H)=\{12, 19, 28, 34, 42, 45, 49, 51\}$.
\end{example}

\begin{example}
  For the tempered monoid $Q$ in Example~\ref{e:quarterstempered},
  ${\mathcal G}(Q)=\{1, 1+\frac{1}{4},1+\frac{1}{2},1+\frac{3}{4},
2+\frac{1}{8},2+\frac{3}{8},2+\frac{5}{8},2+\frac{7}{8},3+\frac{1}{16},3+\frac{3}{16},3+\frac{5}{16},3+\frac{7}{16},3+\frac{9}{16},3+\frac{11}{16},3+\frac{13}{16},3+\frac{15}{16},\dots\}$.
\end{example}\begin{example}
  For the logarithmic monoid $L$ (Example~\ref{ex:fractempered}),
  ${\mathcal G}(L)=\{\log_2(i+1):i\in{\mathbb N}$ and $i+1$ is prime$\}.$
\end{example}

\begin{example}
  For the Pythagorean monoid $P$ (Example~\ref{ex:pitagoras}),
    ${\mathcal G}(P)=\{1,\log_2(3)\}.$
\end{example}

We say that a set is {\em commensurable} if for any pair $\lambda,\mu$ of nonzero elements in the set the fraction $\lambda/\mu$ is rational.
Among all the minimal generating sets of the previous examples we have a variety of sizes and commensurabilities. Indeed,

\begin{itemize}
\item ${\mathcal G}(S)$ and ${\mathcal G}(H)$ are finite and commensurable,
\item ${\mathcal G}(P)$ is finite and not commensurable,
\item ${\mathcal G}(Q)$ is infinite and commensurable,
\item ${\mathcal G}(L)$ is infinite and not commensurable.
\end{itemize}

Our last theorem gives a way to differentiate those $\omega$-monoids that are scalar multiples of numerical semigroups from those that are tempered monoids in terms of the size and commensurability of their minimal generating sets.

\begin{theorem}
\label{t:characterizebygeneratingset}
  Suppose that $\Omega$ is an $\omega$-monoid with minimal generating set ${\mathcal G}(\Omega)$.
  \begin{enumerate}
  \item If ${\mathcal G}(\Omega)$ is infinite, then $\Omega$ is a tempered monoid.
  \item If ${\mathcal G}(\Omega)$ is finite, then $\Omega$ is a tempered monoid if and only if ${\mathcal G}(\Omega)$ is not commensurable.
  \end{enumerate}
\end{theorem}

\begin{proof}
  \begin{enumerate}
  \item The statement follows from Remark~\ref{r:sfinite}.
  \item On one hand, suppose that there exist $\lambda,\mu\in{\mathcal G}(\Omega)$ such that $\lambda/\mu$ is not rational. By Theorem~\ref{t:main}, to prove that $\Omega$ is a tempered monoid, it is enough to prove that there does not exist any real number $\alpha$ such that $\alpha\lambda$ and $\alpha\mu\in{\mathbb N}$. Indeed, suppose that one such $\alpha$ exists.
    Then $\frac{\lambda}{\mu}=\frac{\alpha\lambda}{\alpha\mu}$ is rational, contradicting the hypothesis. Hence, if ${\mathcal G}(\Omega)$ is not commensurable then $\Omega$ is a tempered monoid.

    On the other hand, suppose that
    ${\mathcal G}(\Omega)=\{g_1,\dots,g_n\}$ and that all fractions $\frac{g_i}{g_j}$ are rational, for $1\leq i,j\leq n$. In particular, there exist $\{a_2,a_3,\dots,a_n\}\subseteq{\mathbb N}$ and $\{b_2,b_3,\dots,b_n\}\subseteq{\mathbb N}$ such that $\frac{g_1}{g_i}=\frac{a_i}{b_i}$ for $2\leq i\leq n$ and so
    $g_i=g_1\frac{b_i}{a_i}$.
    Take $\alpha=\frac{a_2\cdot a_3\cdot\dots\cdot a_n}{g_1}$.
    Then it is easy to check that $\alpha g_1\in{\mathbb N}$ and that $\alpha g_i\in{\mathbb N}$ for $2\leq i\leq n$. Consequently $\alpha \Omega\subseteq {\mathbb N}_0$ and so $\Omega$ is a scalar multiple of a numerical semigroup.
\end{enumerate}
  \end{proof}

In \cite{Gottiatomic} Gotti proves that a Puiseux monoid, that is a
submonoid of ${\mathbb Q}_+$, is isomorphic to a numerical monoid if
and only if its generating set is finite.
Since any finite subset of ${\mathbb Q}$ is commensurable, Theorem~\ref{t:characterizebygeneratingset} would establish Gotti's result for the case when the
Puiseux monoid is increasingly enumerable.

\paragraph{Acknowledgments.}
The author was supported by the Spanish government under grant TIN2016-80250-R and by the Catalan government under grant 2014 SGR 537. She would like to acknowledge the interesting conversations with Pilar Bayer, Julio Fernández, Shalom Eliahou, Felix Gotti, Marly Cormar, Scott Chapman, and Alfons Reverté, the constructive and stimulating comments of the anonymous reviewers, and the patience of the editor Susan Jane Colley.

\bibliographystyle{plain}

\begin{thebibliography}{10}

\bibitem{Amiot}
  Amiot, E.
  (2005).
  Rhythmic canons and {G}alois theory.
  In \emph{Colloquium on {M}athematical {M}usic {T}heory}.
  Grazer Math. Ber.
  Vol. 347.
  Graz:
  Karl-Franzens-Univ. Graz,
  pp. 1--33.

\bibitem{Anderson}
  Anderson, D.~D., 
  Anderson, D.~F.,
  Zafrullah, M.
  (1990).
  Factorization in integral domains.
  \emph{J. Pure Appl. Algebra}.
  69 (1): 1--19.
  doi.org/10.1016/0022-4049(90)90074-R.

\bibitem{Andreatta}
  Andreatta, M.
  (2004).
  On group-theoretical methods applied to music: some compositional and implementational aspects.
  In Mazzola, G., Noll, T., Puebla, E., eds.
  \emph{Perspectives in Mathematical and Computational Music Theory}. 
  Osnabrück:
  Electronic Publishing Osnabrück,
  pp. 122--162.

\bibitem{Barbourbook}
  Barbour, J.~M.
  (1951).
  \emph{Tuning and Temperament: A Historical Survey},
  New York:
  Dover Books on Music.

  
\bibitem{ngu}
  Bras-Amor\'os, M.
  (2007).
  On numerical semigroups and their applications to algebraic geometry codes.
  Presented at the {T}hematic {S}eminar ``{A}lgebraic {G}eometry,
  {C}oding and {C}omputing'', {U}niversidad de {V}alladolid, {S}egovia, Spain, October 10.
  \url{www.singacom.uva.es/oldsite/seminarios/WorkshopSG/workshop2/Bras_SG_2007.pdf}.
  
\bibitem{Br:fibonacci}
  Bras-Amor{\'o}s, M.
  (2008).
  Fibonacci-like behavior of the number of numerical semigroups of a given genus.
  \emph{Semigroup Forum}.
  76 (2): 379--384.
  doi.org/10.1007/s00233-007-9014-8.

\bibitem{Bras:MM}
  Bras-Amor\'os, M.
  (submitted).
  Tempered monoids of real numbers, the golden fractal monoid, and the well-tempered harmonic semigroup.
  arxiv.org/abs/1703.01077.
  
\bibitem{Clampitt}
  Clampitt, D.,
  Dom\'inguez, M.,
  Noll, T.
  (2008).
  Well-formed scales, maximally even sets, and {C}hristoffel words.
  In: Klouche, T., Noll, T., eds.
  \emph{MCM 2007. CCIS.}, Vol.~37.
  Heidelberg: Springer.

\bibitem{Clifford}
  Clifford, A.~H.
  (1958).
  Totally ordered commutative semigroups.
  \emph{Bull. Amer. Math.  Soc.}
  64 (6): 305--316.
  doi.org/10.1090/S0002-9904-1958-10221-9.

\bibitem{Costa}
  Costa, F.
  (2015).
  Aphoristic {M}adrigal.
  \url{www.fabiocosta.info/composer.php}.

\bibitem{Crans}
  Crans, A.~S.,  
  Fiore, T.~M.,
  Satyendra, R.
  (2009).
  Musical actions of dihedral groups.
  \emph{Amer. Math. Monthly.}
  116 (6): 479--495.
  doi.org/10.4169/193009709X470399.

\bibitem{Gottiatomic}
  Gotti, F.
  (2017).
  On the atomic structure of {P}uiseux monoids.
  \emph{J. Algebra Appl.}
  16 (7): 1750126, 20.
  doi-org./10.1142/S0219498817501262.

\bibitem{Gotti}
  Gotti, F.
  (2019).
  Increasing positive monoids of ordered fields are {FF}-monoids.
  \emph{J. Algebra}.
  \textbf{518}: 40--56.
  doi.org/10.1016/j.jalgebra.2018.10.010.

\bibitem{GottiGotti}
  Gotti, F., Gotti, M.
  (2018). 
  Atomicity and boundedness of monotone {P}uiseux monoids.
  \emph{Semigroup Forum}
  96 (3): 536--552.
  doi.org/10.1007/s00233-017-9899-9.

\bibitem{Grams}
  Grams, A.
  (1974).
  Atomic rings and the ascending chain condition for principal ideals.
  \emph{Proc. Cambridge Philos. Soc.}
  75 (3): 321--329.
  doi.org/10.1017/S0305004100048532.

\bibitem{HallJosic}
  Hall, R. W., Josi\'{c}, K.
  (2001).
  The mathematics of musical instruments.
  \emph{Amer. Math. Monthly.}
  108 (4): 347--357.
  doi.org/10.2307/2695241.

\bibitem{Harrington}
  Harrington, J.
  The music of {J}eff {H}arrington.
  \url{www.jeffharrington.com}.

\bibitem{Horcik}
  Horc\u{c}\'ik, R.
  (2005).
  Algebraic properties of fuzzy logics.
  PhD dissertation.
  Czech Technical University, Prague, Czechoslovakia.

\bibitem{Kaplan}
  Kaplan, K.
  (2017).
  Counting numerical semigroups.
  \emph{Amer. Math. Monthly.}
  124 (9): 862--875.
  doi.org/10.4169/amer.math.monthly.124.9.862.

\bibitem{Lewin}
  Lewin, D.
  (1987)
  \emph{Generalized musical intervals and transformations.}
  City of Yale:
  Yale University Press.

\bibitem{Neuwirth}
  Neuwirth, E.
  (1997).
  \emph{Musical temperaments}.
  (Steblin, R., trans.)
  Vienna: Springer-Verlag.

\bibitem{Noll}
  Noll, T.
  (2007).
  Musical intervals and special linear transformations.
  \emph{J. Math. Music}.
  1 (2): 121--137.
  doi.org/10.1080/17459730701375026.

\bibitem{Perle}
  Perle, G.
  (1991).
  \emph{Serial composition and atonality}.
  Berkeley and Los Angeles: University of California Press.

\bibitem{Sethares}
  Sethares, W.~A.
  (1999).
  \emph{Tuning, Timbre, Spectrum, Scale},
  London: Springer-Verlag.

\bibitem{Zhai}
  Zhai, A.
  (2013).
  Fibonacci-like growth of numerical semigroups of a given genus.
  \emph{Semigroup Forum}.
  86 (3): 634--662.
  doi.org/10.1007/s00233-012-9456-5.

\end{thebibliography}



\end{document}